\documentclass[12pt,twoside]{amsart}
\usepackage{amsmath, amsthm, amscd, amsfonts, amssymb, graphicx}
\usepackage{enumerate}
\usepackage[colorlinks=true,
linkcolor=blue,
urlcolor=cyan,
citecolor=red]{hyperref}
\usepackage{mathrsfs}
\DeclareFontFamily{U}{matha}{\hyphenchar\font45}
\DeclareFontShape{U}{matha}{m}{n}{
	<5> <6> <7> <8> <9> <10> gen * matha
	<10.95> matha10 <12> <14.4> <17.28> <20.74> <24.88> matha12
}{}
\DeclareSymbolFont{matha}{U}{matha}{m}{n}
\DeclareFontSubstitution{U}{matha}{m}{n}

\DeclareFontFamily{U}{mathx}{\hyphenchar\font45}
\DeclareFontShape{U}{mathx}{m}{n}{
	<5> <6> <7> <8> <9> <10>
	<10.95> <12> <14.4> <17.28> <20.74> <24.88>
	mathx10
}{}
\DeclareSymbolFont{mathx}{U}{mathx}{m}{n}
\DeclareFontSubstitution{U}{mathx}{m}{n}

\DeclareMathDelimiter{\vvvert}{0}{matha}{"7E}{mathx}{"17}
\newtheorem{theorem}{Theorem}[section]
\newtheorem{lemma}{Lemma}[section]
\newtheorem{remark}{Remark}[section]

\newtheorem{corollary}{Corollary}[section]

\numberwithin{equation}{section}

\begin{document}

\title{Improvement and generalization of some Jensen-Mercer-type inequalities}
\author{Hamid Reza Moradi and Shigeru Furuichi}
\subjclass[2010]{Primary 26A51, Secondary 26D15, 26B25}
\keywords{Jensen-Mercer inequality, Jensen inequality, convex function.} \maketitle

\begin{abstract}
The present paper is devoted to the study of Jensen-Mercer-type inequalities. Our results generalize and improve some earlier results in the literature.
\end{abstract}

\pagestyle{myheadings}
\markboth{\centerline {Improvement and generalization of some Jensen-Mercer-type inequalities}}
{\centerline {H.R. Moradi \& S. Furuichi}}
\bigskip
\bigskip
\section{Introduction}
The well-known Jensen inequality for the convex functions states that if $f$ is a convex function on the interval $\left[ m,M \right]$, then
\begin{equation}\label{3}
f\left( \sum\limits_{i=1}^{n}{{{w}_{i}}{{a}_{i}}} \right)\le \sum\limits_{i=1}^{n}{{{w}_{i}}f\left( {{a}_{i}} \right)}
\end{equation}
for all ${{a}_{i}}\in \left[ m,M \right]$ and ${{w}_{i}}\in \left[ 0,1 \right]$ $(i=1,\ldots ,n)$ with $\sum\limits_{i=1}^{n}{{{w}_{i}}}=1$. Various inequalities improving and extending \eqref{3} have been studied in \cite{3,5,1}.

Mercer \cite{4} proved that if $f$ is a convex function on $\left[ m,M \right]$, then
\begin{equation}\label{9}
f\left( M+m-\sum\limits_{i=1}^{n}{{{w}_{i}}{{a}_{i}}} \right)\le f\left( M \right)+f\left( m \right)-\sum\limits_{i=1}^{n}{{{w}_{i}}f\left( {{a}_{i}} \right)}
\end{equation}
for all ${{a}_{i}}\in \left[ m,M \right]$ and ${{w}_{i}}\in \left[ 0,1 \right]$ $(i=1,\ldots ,n)$ with $\sum\limits_{i=1}^{n}{{{w}_{i}}}=1$. Several refinements and generalizations of the inequality \eqref{9} have been
given in \cite{6,2}. 

In \cite[Theorem 2.1]{kian} it has been shown that if $f$ is a convex function on the interval $\left[ m,M \right]$, then
\begin{equation}\label{17}
f\left( M+m-\frac{a+b}{2} \right)\le \frac{1}{b-a}\int_{a}^{b}{f\left( M+m-u \right)du}\le f(M)+f(m)-\frac{f\left( a \right)+f\left( b \right)}{2}
\end{equation}
for all $a,b\in \left[ m,M \right]$. 

In this paper we prove the following general result: 
\begin{equation}\label{13}
f\left( M+m-\overline{a} \right)\le \sum\limits_{i=1}^{n}{\frac{{{w}_{i}}}{\overline{a}-{{a}_{i}}}\int_{M+m-\overline{a}}^{M+m-{{a}_{i}}}{f\left( t \right)dt}}\le f\left( M \right)+f\left( m \right)-\sum\limits_{i=1}^{n}{{{w}_{i}}f\left( {{a}_{i}} \right)}
\end{equation}
where $\bar{a} :=\sum\limits_{i=1}^{n}{{{w}_{i}}{{a}_{i}}} $.  After that,  we show a refinement of inequality \eqref{13} in the following form
\begin{equation*}
\begin{aligned}
f\left( M+m-\overline{a} \right)&\le \sum\limits_{i=1}^{n}{{{w}_{i}}f\left( M+m-\frac{\overline{a}+{{a}_{i}}}{2} \right)} \\ 
& \le \sum\limits_{i=1}^{n}{\frac{{{w}_{i}}}{\overline{a}-{{a}_{i}}}\int_{M+m-\overline{a}}^{M+m-{{a}_{i}}}{f\left( t \right)dt}} \\ 
& \le f\left( M \right)+f\left( m \right)-\sum\limits_{i=1}^{n}{{{w}_{i}}f\left( {{a}_{i}} \right)}.  
\end{aligned}
\end{equation*}
Though we confine our discussion to scalars, by changing the convex function assumption with the operator convex, the inequalities we obtain in this paper can be extended in a natural way to Hilbert space operators.
\section{Main Results}
The following lemma is well-known in \cite[Lemma 1.3]{4}, but we prove it for the reader convenience.
\begin{lemma}\label{1}
Let $f$ be a convex function on $\left[ m,M \right]$, then
\[f\left( M+m-{{a}_{i}} \right)\le f\left( M \right)+f\left( m \right)-f\left( {{a}_{i}} \right),\text{ }\left( m\le {{a}_{i}}\le M,~i=1,\ldots ,n \right).\]
\end{lemma}
\begin{proof}
If $f:\left[ m,M \right]\to \mathbb{R}$ is a convex function, then for any $x,y\in \left[ m,M \right]$ and $t\in \left[ 0,1 \right]$, we have	
\begin{equation}\label{7}
f\left( tx+\left( 1-t \right)y \right)\le tf\left( x \right)+\left( 1-t \right)f\left( y \right).
\end{equation}
It can be verified that if $m\le {{a}_{i}}\le M\left( i=1,\ldots ,n \right)$, then $\frac{M-{{a}_{i}}}{M-m},\frac{{{a}_{i}}-m}{M-m}\le 1$ and $\frac{M-{{a}_{i}}}{M-m}+\frac{{{a}_{i}}-m}{M-m}=1$.  Thanks to \eqref{7}, we have  
\begin{equation}\label{8}
f\left( {{a}_{i}} \right)\le \frac{M-{{a}_{i}}}{M-m}f\left( m \right)+\frac{{{a}_{i}}-m}{M-m}f\left( M \right).
\end{equation}
One the other hand, $m\le {{a}_{i}}\le M\left( i=1,\ldots ,n \right)$ implies $m\le M+m-{{a}_{i}}\le M\left( i=1,\ldots ,n \right)$. Thus, from \eqref{8} we infer
\begin{equation}\label{10}
f\left( M+m-{{a}_{i}} \right)\le \frac{{{a}_{i}}-m}{M-m}f\left( m \right)+\frac{M-{{a}_{i}}}{M-m}f\left( M \right).
\end{equation}
Summing up \eqref{8} and \eqref{10}, we get the desired result.
\end{proof}

Based on this, our first result can be stated as follows:
\begin{theorem}\label{6}
Let $f$ be a convex function on $\left[ m,M \right]$ and $t\in [0,1]$. Then
\begin{equation}\label{5}
\begin{aligned}
 f\left( M+m-\overline{a} \right)&\le \sum\limits_{i=1}^{n}{{{w}_{i}}f\left( M+m-\left( \left( 1-t \right)\overline{a}+t{{a}_{i}} \right) \right)} \\ 
& \le f\left( M \right)+f\left( m \right)-\sum\limits_{i=1}^{n}{{{w}_{i}}f\left( {{a}_{i}} \right)}  
\end{aligned}
\end{equation}
	for all ${{a}_{i}}\in \left[ m,M \right]$ and ${{w}_{i}}\in \left[ 0,1 \right]$ $(i=1,\ldots ,n)$ with $\sum\limits_{i=1}^{n}{{{w}_{i}}}=1$, where $\bar{a} :=\sum\limits_{i=1}^{n}{{{w}_{i}}{{a}_{i}}}$. Moreover, the function $F:\left[ 0,1 \right]\to \mathbb{R}$ defined by
\[F\left( t \right)=\sum\limits_{i=1}^{n}{{{w}_{i}}f\left( M+m-\left( \left( 1-t \right)\overline{a}+t{{a}_{i}} \right) \right)}.\]
	is monotonically nondecreasing and convex on $\left[ 0,1 \right]$.
\end{theorem}
\begin{proof}
Firstly, we have
\begin{equation}\label{16}
\begin{aligned}
\sum\limits_{i=1}^{n}{{{w}_{i}}f\left( M+m-\left( \left( 1-t \right)\overline{a}+t{{a}_{i}} \right) \right)}&\ge f\left( \sum\limits_{i=1}^{n}{{{w}_{i}}\left( M+m-\left( \left( 1-t \right)\overline{a}+t{{a}_{i}} \right) \right)} \right) \\ 
& =f\left( M+m-\overline{a} \right).  
\end{aligned}
\end{equation}
On the other hand, 
\[\begin{aligned}
& \sum\limits_{i=1}^{n}{{{w}_{i}}f\left( M+m-\left( \left( 1-t \right)\overline{a}+t{{a}_{i}} \right) \right)} \\ 
& =\sum\limits_{i=1}^{n}{{{w}_{i}}f\left( \left( 1-t \right)\left( M+m-\overline{a} \right)+t\left( M+m-{{a}_{i}} \right) \right)} \\ 
& \le \sum\limits_{i=1}^{n}{{{w}_{i}}\left( \left( 1-t \right)f\left( M+m-\overline{a} \right)+tf\left( M+m-{{a}_{i}} \right) \right)} \\ 
& \le \sum\limits_{i=1}^{n}{{{w}_{i}}\left( \left( 1-t \right)\left( f\left( M \right)+f\left( m \right)-\sum\limits_{j=1}^{n}{{{w}_{j}}f\left( {{a}_{j}} \right)} \right)+t\left( f\left( M \right)+f\left( m \right)-f\left( {{a}_{i}} \right) \right) \right)} \\ 
& =f\left( M \right)+f\left( m \right)-\sum\limits_{i=1}^{n}{{{w}_{i}}f\left( {{a}_{i}} \right).} \\ 
\end{aligned}\]
For the convexity of $F$, we have
\[\begin{aligned}
& F\left( \frac{t+s}{2} \right) \\ 
& =\sum\limits_{i=1}^{n}{{{w}_{i}}f\left( M+m-\left( \left( 1-\frac{t+s}{2} \right)\overline{a}+\frac{t+s}{2}{{a}_{i}} \right) \right)} \\ 
& =\sum\limits_{i=1}^{n}{{{w}_{i}}f\left( M+m-\left( \frac{\left( 1-t \right)\overline{a}+t{{a}_{i}}+\left( 1-s \right)\overline{a}+s{{a}_{i}}}{2} \right) \right)} \\ 
& =\sum\limits_{i=1}^{n}{{{w}_{i}}f\left( \frac{M+m-\left( \left( 1-t \right)\overline{a}+t{{a}_{i}} \right)+M+m-\left( \left( 1-s \right)\overline{a}+s{{a}_{i}} \right)}{2} \right)} \\ 
& \le \frac{1}{2}\left[ \sum\limits_{i=1}^{n}{{{w}_{i}}f\left( M+m-\left( \left( 1-t \right)\overline{a}+t{{a}_{i}} \right) \right)}+\sum\limits_{i=1}^{n}{{{w}_{i}}f\left( M+m-\left( \left( 1-s \right)\overline{a}+s{{a}_{i}} \right) \right)} \right] \\ 
& =\frac{F\left( t \right)+F\left( s \right)}{2}. \\ 
\end{aligned}\]
Now, if $0<s<t<1$, then $s=\frac{t-s}{t}\cdot0+\frac{s}{t}\cdot t$ and hence the convexity of $F$ implies
	\[\begin{aligned}
	F\left( s \right)&=F\left( \frac{t-s}{t}\cdot0+\frac{s}{t}\cdot t \right) \\ 
	& \le \frac{t-s}{t}F\left( 0 \right)+\frac{s}{t}F\left( t \right) \\ 
	& \le \frac{t-s}{t}F\left( t \right)+\frac{s}{t}F\left( t \right) \\ 
	& =F\left( t \right).  
	\end{aligned}\]
We remark that the second inequality in the above follows from \eqref{16} and the fact
\[F\left( 0 \right)=\sum\limits_{i=1}^{n}{{{w}_{i}}f\left( M+m-\overline{a} \right)}=f\left( M+m-\overline{a} \right).\]
Therefore $F$ is monotonically nondecreasing on $\left[ 0,1 \right]$.
\end{proof}

\begin{corollary}\label{14}
Let all the assumptions of Theorem \ref{6} hold, then
\begin{equation*}
f\left( M+m-\overline{a} \right)\le \sum\limits_{i=1}^{n}{\frac{{{w}_{i}}}{\overline{a}-{{a}_{i}}}\int_{M+m-\overline{a}}^{M+m-{{a}_{i}}}{f\left( t \right)dt}}\le f\left( M \right)+f\left( m \right)-\sum\limits_{i=1}^{n}{{{w}_{i}}f\left( {{a}_{i}} \right)}.	
\end{equation*}
\end{corollary}
\begin{proof}
Integrating the inequality \eqref{5} over $t\in \left[ 0,1 \right]$, we get \eqref{14}. Here we used the fact
\[\begin{aligned}
 \int_{0}^{1}{f\left( M+m-\left( \left( 1-t \right)\overline{a}+t{{a}_{i}} \right) \right)dt}&=\int_{0}^{1}{f\left( \left( 1-t \right)\left( M+m-\overline{a} \right)+t\left( M+m-{{a}_{i}} \right) \right)dt} \\ 
& =\int_{0}^{1}{f\left( t\left( M+m-\overline{a} \right)+\left( 1-t \right)\left( M+m-{{a}_{i}} \right) \right)dt} \\ 
& =\frac{1}{\overline{a}-{{a}_{i}}}\int_{M+m-\overline{a}}^{M+m-{{a}_{i}}}{f\left( t \right)dt}.  
\end{aligned}\]
\end{proof}

\begin{remark}
	Put $n=2$, ${{w}_{1}}={{w}_{2}}={1}/{2}\;$, ${{a}_{1}}=a$, and ${{a}_{2}}=b$ in Corollary \ref{14}, then
	\[\begin{aligned}
	f\left( M+m-\frac{a+b}{2} \right)& \le \frac{1}{b-a}\int_{a}^{b}{f\left( M+m-u \right)du} \\ 
	& \le f\left( M \right)+f\left( m \right)-\frac{f\left( a \right)+f\left( b \right)}{2}  
	\end{aligned}\] 
which shows that our inequality \eqref{14}  generalizes inequality \eqref{17}.
\end{remark}
We give a more precise estimate in the next theorem.
\begin{theorem}\label{15}
Let $f$ be a convex function on $\left[ m,M \right]$. Then
\begin{equation}\label{4}
\begin{aligned}
 f\left( M+m-\overline{a} \right)&\le \sum\limits_{i=1}^{n}{{{w}_{i}}f\left( M+m-\frac{\overline{a}+{{a}_{i}}}{2} \right)} \\ 
& \le \sum\limits_{i=1}^{n}{\frac{{{w}_{i}}}{\overline{a}-{{a}_{i}}}\int_{M+m-\overline{a}}^{M+m-{{a}_{i}}}{f\left( t \right)dt}} \\ 
& \le f\left( M \right)+f\left( m \right)-\sum\limits_{i=1}^{n}{{{w}_{i}}f\left( {{a}_{i}} \right)}  
\end{aligned}
\end{equation}
for all ${{a}_{i}}\in \left[ m,M \right]$ and ${{w}_{i}}\in \left[ 0,1 \right]$ $(i=1,\ldots ,n)$ with $\sum\limits_{i=1}^{n}{{{w}_{i}}}=1$.
\end{theorem}
\begin{proof}
If $f:\left[ m,M \right]\to \mathbb{R}$ is a convex function, then we have for any $a,b\in \left[ m,M \right]$
\[\begin{aligned}
 f\left( \frac{a+b}{2} \right)&=f\left( \frac{ta+\left( 1-t \right)b+tb+\left( 1-t \right)a}{2} \right) \\ 
& \le \frac{f\left( ta+\left( 1-t \right)b \right)+f\left( tb+\left( 1-t \right)a \right)}{2} \\ 
& \le \frac{f\left( a \right)+f\left( b \right)}{2}.  
\end{aligned}\] 
Replacing $a$ and $b$ by $M+m-a$ and $M+m-b$, respectively, we get
\[\begin{aligned}
 f\left( M+m-\frac{a+b}{2} \right)&\le \frac{f\left( M+m-\left( ta+\left( 1-t \right)b \right) \right)+f\left( M+m-\left( tb+\left( 1-t \right)a \right) \right)}{2} \\ 
& \le \frac{f\left( M+m-a \right)+f\left( M+m-b \right)}{2}.  
\end{aligned}\] 
Integrating the inequality over $t\in \left[ 0,1 \right]$, and using the fact
\[\int_{0}^{1}{f\left( tx+\left( 1-t \right)y \right)dt}=\int_{0}^{1}{f\left( ty+\left( 1-t \right)x \right)dt},\]
we infer that
\[\begin{aligned}
 f\left( M+m-\frac{a+b}{2} \right)&\le \int_{0}^{1}{f\left( M+m-\left( ta+\left( 1-t \right)b \right) \right)dt} \\ 
& \le \frac{f\left( M+m-a \right)+f\left( M+m-b \right)}{2}.  
\end{aligned}\]
Since ${{a}_{i}}, \bar{a} \in \left[ m,M \right]$, we can write
\[\begin{aligned}
 f\left( M+m-\frac{\overline{a}+{{a}_{i}}}{2} \right)&\le \frac{1}{\overline{a}-{{a}_{i}}}\int_{M+m-\overline{a}}^{M+m-{{a}_{i}}}{f\left( t \right)dt} \\ 
& \le \frac{f\left( M+m-\overline{a} \right)+f\left( M+m-{{a}_{i}} \right)}{2},  
\end{aligned}\]
due to
\[\int_{0}^{1}{f\left( M+m-\left( t\overline{a}+\left( 1-t \right){{a}_{i}} \right) \right)dt}=\frac{1}{\overline{a}-{{a}_{i}}}\int_{M+m-\overline{a}}^{M+m-{{a}_{i}}}{f\left( t \right)dt}.\]
Multiplying by ${{w}_{i}}>0$ $\left( i=1,\ldots ,n \right)$ and summing over $i$ from $1$ to $n$ we may deduce 
\begin{equation}\label{2}
\begin{aligned}
 \sum\limits_{i=1}^{n}{{{w}_{i}}f\left( M+m-\frac{\overline{a}+{{a}_{i}}}{2} \right)}&\le \sum\limits_{i=1}^{n}{\frac{{{w}_{i}}}{\overline{a}-{{a}_{i}}}\int_{M+m-\overline{a}}^{M+m-{{a}_{i}}}{f\left( t \right)dt}} \\ 
& \le \frac{f\left( M+m-\overline{a} \right)+\sum\nolimits_{i=1}^{n}{{{w}_{i}}f\left( M+m-{{a}_{i}} \right)}}{2}.  
\end{aligned}
\end{equation}
On the other hand, by \eqref{3}
\begin{equation}\label{11}
\begin{aligned}
 f\left( M+m-\bar{a} \right)&=f\left( \sum\limits_{i=1}^{n}{{{w}_{i}}\left( M+m-\frac{\bar{a}+{{a}_{i}}}{2} \right)} \right) \\ 
& \le \sum\limits_{i=1}^{n}{{{w}_{i}}f\left( M+m-\frac{\bar{a}+{{a}_{i}}}{2} \right)}  
\end{aligned}
\end{equation}
and by Lemma \ref{1}
\begin{equation}\label{12}
\begin{aligned}
& \frac{f\left( M+m-\bar{a} \right)+{\sum\nolimits_{i=1}^{n}{{{w}_{i}}f\left( M+m-{{a}_{i}} \right)}}}{2} \\ 
& \le \frac{f\left( M \right)+f\left( m \right)-\sum\nolimits_{j=1}^{n}{{{w}_{j}}f\left( {{a}_{j}} \right)}+{f\left( M \right)+f\left( m \right)-\sum\nolimits_{i=1}^{n}{{{w}_{i}}f\left( {{a}_{i}} \right)}}}{2} \\ 
& =f\left( M \right)+f\left( m \right)-\sum\limits_{i=1}^{n}{{{w}_{i}}f\left( {{a}_{i}} \right)}.
\end{aligned}
\end{equation}
Combining \eqref{2}, \eqref{11}, and \eqref{12}, we get \eqref{4}.
\end{proof}

\begin{corollary}\label{18}
Let ${{a}_{i}}\in \left[ m,M \right]$ and ${{w}_{i}}\in \left[ 0,1 \right]$ $(i=1,\ldots ,n)$ with $\sum\limits_{i=1}^{n}{{{w}_{i}}}=1$. Then
\[\begin{aligned}
 \frac{Mm}{\prod\limits_{i=1}^{n}{a_{i}^{{{w}_{i}}}}}&\le \exp \left[ \sum\limits_{i=1}^{n}{\frac{{{w}_{i}}}{\overline{a}-{{a}_{i}}}\int_{M+m-\overline{a}}^{M+m-{{a}_{i}}}{\log tdt}} \right] \\ 
& \le \prod\limits_{i=1}^{n}{{{\left( M+m-\frac{\overline{a}+{{a}_{i}}}{2} \right)}^{{{w}_{i}}}}} \\ 
& \le M+m-\sum\limits_{i=1}^{n}{{{w}_{i}}{{a}_{i}}}.  
\end{aligned}\]
\end{corollary}

\begin{proof}
Put $f(t)=-\log t$, $(0<t \le 1)$ in Theorem \ref{15}.
\end{proof}

\begin{remark}
	If we set $n=2$, $a_1=m,a_2=M$ and $w_1=w_2=1/2$ in Corollary \ref{18}, then we have
	$$
	\sqrt{Mm} \leq \frac{M^{\frac{M}{M-m}}}{e m^{\frac{m}{M-m}}} \leq \frac{1}{4}\sqrt{(M+3m)(m+3M)} \leq \frac{1}{2}\left(M+m\right).
	$$
One can  obtain the inequalities for the weighted parameter in means of two variables $m$ and $M$ by elementary calculations. We leave it to the interested readers. 
\end{remark}

\begin{remark}
Let all the assumptions of Theorem \ref{15} hold, then
\[\begin{aligned}
 f\left( \overline{a} \right)&\le \sum\limits_{i=1}^{n}{{{w}_{i}}f\left( \frac{\overline{a}+{{a}_{i}}}{2} \right)} \\ 
& \le \sum\limits_{i=1}^{n}{\frac{{{w}_{i}}}{{{a}_{i}}-\overline{a}} \int_{\overline{a}}^{{{a}_{i}}}{f\left( t \right)dt}} \\ 
& \le \frac{f\left( \overline{a} \right)+\sum\nolimits_{i=1}^{n}{{{w}_{i}}f\left( {{a}_{i}} \right)}}{2} \\ 
& \le \sum\limits_{i=1}^{n}{{{w}_{i}}f\left( {{a}_{i}} \right)}.  
\end{aligned}\]
The proof is in the same spirit as that of Theorem \ref{15} (see also \cite[Corollary 3]{dragomir}).
\end{remark}

\section*{Acknowledgements}
The author (S.F.) was partially supported by JSPS KAKENHI Grant Number 16K05257.

{\tiny \vskip 0.3 true cm }

{\tiny (H.R. Moradi) Department of Mathematics, Payame Noor University (PNU), P.O. Box 19395-4697, Tehran, Iran.}

{\tiny \textit{E-mail address:} hrmoradi@mshdiau.ac.ir }

{\tiny \vskip 0.3 true cm }

{\tiny (S. Furuichi) Department of Information Science, College of Humanities and Sciences, Nihon University, 3-25-40, Sakurajyousui, Setagaya-ku, Tokyo, 156-8550, Japan.}

{\tiny \textit{E-mail address:} furuichi@chs.nihon-u.ac.jp}
\end{document}